\documentclass[10pt,reqno]{amsart}
\usepackage{pdfsync}
\usepackage{amssymb,amscd}   

\newtheorem{theorem}{Theorem}
\newtheorem{proposition}[theorem]{Proposition} 
\newtheorem{lemma}[theorem]{Lemma}

\begin{document}
\title[Cartan subalgebras in continuous cores]{Absence of Cartan subalgebras in continuous cores of free product von Neumann algebras}
\author[Y. Ueda]
{Yoshimichi UEDA}
\address{
Graduate School of Mathematics, 
Kyushu University, 
Fukuoka, 819-0395, Japan
}
\email{ueda@math.kyushu-u.ac.jp}
\thanks{Supported by Grant-in-Aid for Scientific Research (C) 24540214.}
\begin{abstract} 
We show that the continuous core of any type III free product factor has no Cartan subalgebra. This is a complement to previous works due to Houdayer--Ricard and Boutonnet--Houdayer--Raum. 
\end{abstract}

\maketitle

\allowdisplaybreaks{

\section{Introduction and Statement}

It is known, see \cite[Theorem 1]{FeldmanMoore:TAMS77}, that a given von Neumann algebra comes from an orbit equivalence relation if and only if it has a Cartan subalgebra (i.e., a MASA with normal conditional expectation, whose normalizer generates the whole algebra). Hence the search for Cartan subalgebras in a given von Neumann algebra is thought of as that for hidden dynamical systems producing the algebra. Therefore, it is important in view of ergodic theory to seek for Cartan subalgebras in a given von Neumann algebra.

The aim of this short note is to complement two recent important works on free product von Neumann algebras, due to Houdayer--Ricard \cite{HoudayerRicard:AdvMath11} and Boutonnet--Houdayer--Raum \cite{BoutonnetHoudayerRaum:Preprint12} establishing, among others, that {\it any} free product von Neumann algebra has no Cartan subalgebra. The work \cite{BoutonnetHoudayerRaum:Preprint12} generalizes Ioana's previous important work \cite{Ioana:Preprint12} on type II$_1$ factors to arbitrary factors. In the 90s Voiculescu  \cite{Voiculescu:GAFA96} first proved the absence of Cartan subalgebras (or more generally diffuse hyperfinite regular subalgebras) in free group factors by free entropy. In the late 00s Ozawa and Popa \cite{OzawaPopa:AnnMath10} succeeded in proving its various improved assertions by deformation/rigidity and intertwining techniques, and a couple of years later Popa and Vaes made an epoch-making work \cite{PopaVaes:ActaMath14} in the direction. The works \cite{HoudayerRicard:AdvMath11},\cite{Ioana:Preprint12},\cite{BoutonnetHoudayerRaum:Preprint12} (and thus this note too) were done under the influence of these two breakthroughs \cite{OzawaPopa:AnnMath10},\cite{PopaVaes:ActaMath14}. See \cite[\S1]{BoutonnetHoudayerRaum:Preprint12} for further historical remarks in the direction. 

\medskip
Let $M_1, M_2$ be two non-trivial (i.e., $\neq \mathbb{C}$) von Neumann algebras with separable preduals, and $\varphi_1, \varphi_2$ be faithful normal states on them, respectively. Denote by $(M,\varphi) = (M_1,\varphi_1)\star(M_2,\varphi_2)$ their free product (see e.g.~\cite[\S\S2.1]{Ueda:AdvMath11}) throughout this note. By \cite[Theorem 4.1]{Ueda:AdvMath11} the free product von Neumann algebra $M$ admits the following general structure: $M = M_d\oplus M_c$ with finite dimensional $M_d$ and diffuse $M_c$ such that $M_d$ can explicitly be calculated with possibly $M=M_c$, and moreover, such that if $(\mathrm{dim}(M_1),\mathrm{dim}(M_2)) \neq (2,2)$, then $M_c$ becomes a full factor of type II$_1$ or III$_\lambda$ ($\lambda \neq 0$) and the T-set $T(M_c)$ does the kernel of the modular action $t \in \mathbb{R} \mapsto \sigma_t^\varphi = \sigma_t^{\varphi_1}\star\sigma_t^{\varphi_2} \in \mathrm{Aut}(M)$ itself; otherwise $M_c = L^\infty[0,1]\,\bar{\otimes}\,M_2(\mathbb{C})$. This note is mainly devoted to establishing the following: 

\begin{theorem}\label{T1} If $M_c$ is of type III, then its continuous core $\widetilde{M_c} = M_c\rtimes_{\sigma^{\varphi_c}}\mathbb{R}$ {\rm(}i.e., the crossed product of $M_c$ by the modular action $t \in \mathbb{R} \mapsto \sigma_t^{\varphi_c} \in \mathrm{Aut}(M_c)${\rm)} with $\varphi_c := \varphi\!\upharpoonright_{M_c}$ does never possess any Cartan subalgebra. 
\end{theorem}

Our motivations are as follows. It is well-known that if $A$ is a Cartan subalgebra in a von Neumann algebra $N$, then so is $A\rtimes_{\sigma^\psi}\mathbb{R} = A\,\bar{\otimes}\,L(\mathbb{R})$ in the continuous core $\widetilde{N} = N\rtimes_{\sigma^\psi}\mathbb{R}$ with a faithful normal state $\psi = \psi\circ E$ on $N$, where $E$ denotes the unique normal conditional expectation from $N$ onto $A$. The work \cite{BoutonnetHoudayerRaum:Preprint12} actually shows only the absence of such special Cartan subalgebras in the continuous core $\widetilde{M}$ (even when both $M_1, M_2$ are hyperfinite). Hence Theorem \ref{T1} is seemingly stronger than the original one \cite[Theorem A]{BoutonnetHoudayerRaum:Preprint12}. On the other hand, the former work \cite{HoudayerRicard:AdvMath11} shows the absence of general Cartan subalgebras in the continuous core of any type III$_1$ free Araki--Woods factor; thus the question that Theorem \ref{T1} answers affirmatively was a simple test of the problem \cite[\S\S5.4]{Ueda:AdvMath11} asking whether or not $M_c$ falls into the class of free Araki--Woods factors introduced by Shlyakhtenko \cite{Shlyakhtenko:PacificJMath97} when both $M_1, M_2$ are hyperfinite. (See \S3 for other tests.) Next, it is known that the continuous core $\widetilde{M}$ is an amalgamated free product von Neumann algebra over a \emph{diffuse} subalgebra (see \cite[Theorem 5.1]{Ueda:PacificJMath99}), and amalgamated free products over diffuse subalgebras usually behave quite differently from plain free products. Lastly, the structure theory for type III factors (see \cite[Ch.XII]{Takesaki:Book}) suggests that a preferable way of study of type III factors is to regard their continuous cores (or more preferably their discrete cores if possible) with canonical group actions as main objects rather than their associates. 

Theorem \ref{T1} and \cite[\S\S2.4]{Ueda:Preprint12-13} altogether show that the continuous/discrete cores of any `type III free product factor' have no Cartan subalgebra. Consequently, this note completes the study of proving the absence of Cartan subalgebras for arbitrary free product von Neumann algebras, though what is new is the combination of technologies provided in \cite{HoudayerRicard:AdvMath11},\cite{BoutonnetHoudayerRaum:Preprint12} with Proposition \ref{P2} provided below.    

\section{Proof} 

Keep the notation in \S1. As emphasized in \S1 the question here is about an amalgamated free product von Neumann algebra over a diffuse subalgebra. The amalgamated free product in question arises from the inclusions $M_i \rtimes_{\sigma^{\varphi_i}}\mathbb{R} \supseteq \mathbb{C}1\rtimes_{\sigma^{\varphi_i}}\mathbb{R}$, $i=1,2$, see \cite[Theorem 5.1]{Ueda:PacificJMath99}. Hence the next simple observation, which itself is of independent interest, plays a key r\^{o}le in our discussion below. In what follows, for a given (unital) inclusion $P \supseteq Q$ of von Neumann algebras we denote by $\mathcal{N}_P(Q)$ the normalizer of $Q$ in $P$, i.e., all unitaries $u \in P$ with $uQu^* = Q$. 

\begin{proposition}\label{P2} Let $N$ be a von Neumann algebra with separable predual and $\psi$ be a faithful normal positive linear functional on it. Then the normalizer $\mathcal{N}_{\widetilde{N}}(\mathbb{C}1\rtimes_{\sigma^{\psi}}\mathbb{R})$ of $\mathbb{C}1\rtimes_{\sigma^{\psi}}\mathbb{R} = \mathbb{C}1\,\bar{\otimes}\,L(\mathbb{R})$ in $\widetilde{N} = N\rtimes_{\sigma^{\psi}}\mathbb{R}$ sits inside $N_\psi\rtimes_{\sigma^{\psi}}\mathbb{R} = N_\psi\,\bar{\otimes}\,L(\mathbb{R})$, and hence $\mathcal{N}_{\widetilde{N}}(\mathbb{C}1\rtimes_{\sigma^{\psi}}\mathbb{R})$ is exactly the unitary group of $N_\psi\,\bar{\otimes}\,L(\mathbb{R})$. In particular, if the centralizer $N_\psi$ is trivial, then $\mathbb{C}1\rtimes_{\sigma^{\psi}}\mathbb{R}$ is a singular MASA in $\widetilde{N}$. 
\end{proposition}   
\begin{proof} The discussion below follows the idea of the proof of \cite[Theorem 2.1]{NeshveyevStormer:JFA02}, but the key is the so-called modular condition instead. Let $\rho : \mathbb{R} \curvearrowright L^2(\mathbb{R})$ be the `right' regular representation, i.e., $\rho_t = \lambda_{-t}$, $t \in \mathbb{R}$, with the usual `left' regular representation $\lambda : \mathbb{R} \curvearrowright L^2(\mathbb{R})$. It is standard, see e.g.~\cite[Theorem 3.11]{vanDaele:LMSLectureNoteSer31}, that $N\rtimes_{\sigma^{\psi}}\mathbb{R} \supseteq \mathbb{C}1\rtimes_{\sigma^{\psi}}\mathbb{R}$ is identical to 
$(N\,\bar{\otimes}\,B(L^2(\mathbb{R})))^{(\sigma^\psi\,\bar{\otimes}\,\mathrm{Ad}\rho,\mathbb{R})} \supseteq \mathbb{C}1\,\bar{\otimes}\,L(\mathbb{R})$, which is conjugate to 
\begin{equation}\label{Eq1}
(N\,\bar{\otimes}\,B(L^2(\mathbb{R})))^{(\sigma^\psi\,\bar{\otimes}\,\mathrm{Ad}v,\mathbb{R})} \supseteq \mathbb{C}1\,\bar{\otimes}\,L^\infty(\mathbb{R}) 
\end{equation}
by taking the Fourier transform on the second component, where $L^\infty(\mathbb{R})$ acts on $L^2(\mathbb{R})$ by multiplication and the $v_t$, $t \in \mathbb{R}$, are the unitary elements in $L^\infty(\mathbb{R})$ defined to be $v_t(s) := e^{its}$, $s \in \mathbb{R}$. Hence it suffices to work with the inclusion \eqref{Eq1} instead of the original inclusion.

Let $u \in (N\,\bar{\otimes}\,B(L^2(\mathbb{R})))^{(\sigma^\psi\,\bar{\otimes}\,\mathrm{Ad}v,\mathbb{R})}$ be a unitary element such that $u(\mathbb{C}1\,\bar{\otimes}\,L^\infty(\mathbb{R}))u^* = \mathbb{C}1\,\bar{\otimes}\,L^\infty(\mathbb{R})$. By \cite[Appendix IV]{Dixmier:Book} (together with a trick used in the proof of \cite[Theorem 17.41]{Kechris:GTM156} if necessary) one can choose a non-singular Borel bijection $\alpha : \mathbb{R} \to \mathbb{R}$ in such a way that $u(1\otimes f)u^* = 1\otimes(f\circ\alpha^{-1}) = 1\otimes u_\alpha f u_\alpha^*$ for every $f \in L^\infty(\mathbb{R})$, where $(u_\alpha g)(s) = [(dm\circ\alpha^{-1}/dm)(s)]^{1/2} g(\alpha^{-1}(s))$, $g \in L^2(\mathbb{R})$ with the Lebesgue measure $m(ds) = ds$. Set $w := u(1\otimes u_\alpha^*)$, a unitary element in $(N\,\bar{\otimes}\,B(L^2(\mathbb{R}))) \cap (\mathbb{C}1\,\bar{\otimes}\,L^\infty(\mathbb{R}))' = N\,\bar{\otimes}\,L^\infty(\mathbb{R})$ by \cite[Theorem IV.5.9; Corollary IV.5.10]{Takesaki:Book}, since $L^\infty(\mathbb{R})$ is a MASA in $B(L^2(\mathbb{R}))$. Since $(\sigma_t^\psi\,\bar{\otimes}\,\mathrm{Ad}v_t)(u) = u$, for every $t \in \mathbb{R}$ one has $w = (\sigma_t^\psi\,\bar{\otimes}\,\mathrm{id})(w) (1\otimes e^{it((\,\cdot\,)-\alpha^{-1}(\,\cdot\,))})$; hence 
\begin{equation}\label{Eq2} 
(\sigma_t^\psi\,\bar{\otimes}\,\mathrm{id})(w) = (1\otimes e^{it(\alpha^{-1}(\,\cdot\,)-(\,\cdot\,))})w.   
\end{equation}
Since the standard Hilbert space $\mathcal{H}:=L^2(N)$ is separable (see e.g.~\cite[Lemma 1.8]{Yamagami:JMSJ59}), we can appeal to the disintegration 
\begin{equation}\label{Eq3}
\mathcal{H}\,\bar{\otimes}\,L^2(\mathbb{R}) = \int_\mathbb{R}^\oplus \mathcal{H}(s)\,ds, \quad N\,\bar{\otimes}\,L^\infty(\mathbb{R}) = \int_\mathbb{R}^\oplus N(s)\,ds  
\end{equation} 
with the constant fields $\mathcal{H}(s) = \mathcal{H}$, $N(s) = N$ (see e.g.~\cite[Part II, Ch.~3, \S4; Corollary of Proposition 3]{Dixmier:Book}). Thus we can write $w = \int_\mathbb{R}^\oplus w(s)\,ds$ and choose $s \mapsto w(s)$ as a measurable field of unitary elements in $N$ (see e.g.~\cite[Part II, Ch.~2, p.183]{Dixmier:Book}). By the identification \eqref{Eq3} we observe that 
$$
(\sigma_t^\psi\,\bar{\otimes}\,\mathrm{id})(w) = (\Delta_\psi^{it}\otimes1)w(\Delta_\psi^{-it}\otimes1) = \int_\mathbb{R}^\oplus \Delta_\psi^{it}w(s)\Delta_\psi^{-it}\,ds = \int_\mathbb{R}^\oplus \sigma_t^\psi(w(s))\,ds, 
$$
where $\Delta_\psi$ is the modular operator associated with $\psi$. Therefore, the identity \eqref{Eq2} is translated into 
$$
\int_\mathbb{R}^\oplus \sigma_t^\psi(w(s))\,ds = \int_\mathbb{R}^\oplus e^{it(\alpha^{-1}(s)-s)}w(s)\,ds. 
$$
This implies, by e.g.~\cite[Part II, Ch.~2, \S3, Corollary of Proposition 2]{Dixmier:Book}, that there exists a co-null subset $S$ of $\mathbb{R}$ so that for every $s \in S$ one has $\sigma_t^\psi(w(s)) = e^{it(\alpha^{-1}(s)-s)}w(s)$ for all rational numbers $t$ and hence for all $t \in \mathbb{R}$ by continuity. Therefore, for all $s \in S$ $t \mapsto \sigma_t^\psi(w(s))$ has an entire extension $\sigma_z^\psi(w(s)) := e^{iz(\alpha^{-1}(s)-s)}w(s)$ so that by the modular condition, e.g.~\cite[Exercise VIII.2.2]{Takesaki:Book}, we have 
$$
\psi(1) = \psi(w(s)^* w(s)) = \psi(\sigma_i^\psi(w(s))w(s)^*) = e^{(s-\alpha^{-1}(s))}\psi(w(s)w(s)^*) = e^{(s-\alpha^{-1}(s))}\psi(1), 
$$
implying that $\alpha^{-1}(s) = s$ and $w(s) \in N_\psi$. Thanks to \cite[Part II, Ch.~3, \S1, Theorem 1]{Dixmier:Book} we conclude that 
$$
u = w = \int_\mathbb{R}^\oplus w(s)\,ds \in \int_\mathbb{R}^\oplus N_\psi(s)\,ds = N_\psi\,\bar{\otimes}\,L^\infty(\mathbb{R})
$$
with the constant field $N_\psi(s) = N_\psi$. This immediately implies the desired assertion.    
\end{proof}

Let us start proving Theorem \ref{T1}. In what follows, $\mathrm{Tr}$ stands for the canonical trace on $\widetilde{M} = M\rtimes_{\sigma^\varphi}\mathbb{R}$, which the so-called dual action scales. (See \cite[Ch.~XII]{Takesaki:Book}.) We need to recall a central notation of intertwining techniques, initiated by Popa, in the present setup. Let $P, Q$ be (not necessarily unital) von Neumann subalgebras of $\widetilde{M}$ such that $\mathrm{Tr}(1_P)$ is finite and $\mathrm{Tr}\!\upharpoonright_Q$ still semifinite. We write $P \preceq_{\widetilde{M}} Q$ if there exist a non-zero projection $q \in Q$ with $\mathrm{Tr}(q)$ finite, a natural number $n$, a (possibly non-unital) normal $*$-homomorphism $\pi$ from $P$ into the $n\times n$ matrices over $qQq$, and a non-zero partial isometry $y$ as a $1\times n$ matrix over $1_P\widetilde{M}q$ such that $xy = y\pi(x)$ holds for every $x \in P$. See \cite[Lemma 2.2]{HoudayerRicard:AdvMath11},\cite[Lemma 2.3]{BoutonnetHoudayerRaum:Preprint12} (due to Vaes) for its equivalent conditions. The main, necessary ingredients from \cite{HoudayerRicard:AdvMath11},\cite{BoutonnetHoudayerRaum:Preprint12} are in order. 
 
 \medskip
(I) {\it Assume that both $M_1, M_2$ are hyperfinite.} As pointed out in \cite[\S\S5.2]{HoudayerRicard:AdvMath11} Theorem 5.2 of the same paper still holds true in the present setup. The proof is basically same, and finally arrives at the main argument in the proof of \cite[Theorem 3.5]{HoudayerShlyakhtenko:IMRN11} re-organizing several arguments from \cite{OzawaPopa:AnnMath10},\cite{OzawaPopa:AmerJMath10}. However, one has to replace \cite[Theorem A]{HoudayerRicard:AdvMath11} and \cite[Theorem 4.3]{HoudayerRicard:AdvMath11} (with the `free malleable deformation') by \cite[Theorem 4.8]{RicardXu:Crelle06} and the proof of \cite[Theorem 4.2]{ChifanHoudayer:DukeMathJ10} 
(with Ioana--Peterson--Popa's original malleable deformation), respectively. Hence the consequence becomes as follows. Let $p \in \mathbb{C}1\rtimes_{\sigma^{\varphi}}\mathbb{R}$ be a non-zero projection with $\mathrm{Tr}(p)$ finite. Let $P \subseteq p\widetilde{M}p$ be a (not necessarily unital) hyperfinite von Neumann subalgebra. If $P \not\preceq_{\widetilde{M}} \widetilde{M_i} = M_i \rtimes_{\sigma^{\varphi_i}}\mathbb{R}$ ($\hookrightarrow \widetilde{M}$ canonically) for all $i$, then $\mathcal{N}_{1_P \widetilde{M}1_P}(P)''$ is hyperfinite. (This statement itself does not use the full power of the assumption that both $M_1, M_2$ are hyperfinite.) Then, \cite[Proposition 2.7]{BoutonnetHoudayerRaum:Preprint12} with the hyperfiniteness of $M_1, M_2$ gives the necessary consequence: {\it If $\mathcal{N}_{1_P\widetilde{M}1_P}(P)''$ has no hyperfinite direct summand, then $P\preceq_{\widetilde{M}} \mathbb{C}1\rtimes_{\sigma^\varphi}\mathbb{R}$.}  

(II) {\it Assume that either $M_1$ or $M_2$ has no hyperfinite direct summand. Let $p \in \mathbb{C}1\rtimes_{\sigma^{\varphi}}\mathbb{R}$ be a non-zero projection with $\mathrm{Tr}(p)$ finite, and let $P$ be a unital regular hyperfinite von Neumann subalgebra in $p\widetilde{M}p$. Then, \cite[Proposition 2.8, Theorem 5.1, Lemma 5.2]{BoutonnetHoudayerRaum:Preprint12} altogether show that $P \preceq_{\widetilde{M}} \mathbb{C}1\rtimes_{\sigma^\varphi}\mathbb{R}$.} 

\medskip
As explained in \cite[\S\S2.1]{Ueda:Preprint12-13} we may and do assume $M = M_c$ after cutting $M$ by a suitable central projection of either $M_1$ or $M_2$ if necessary. Moreover, when $M_i$ is not hyperfinite, the same trick enables us to assume that $M_i$ indeed has no hyperfinite direct summand.  

\medskip
Suppose, on the contrary, that there exists a Cartan subalgebra $Q$ in $\widetilde{M}$. Let $q \in \mathbb{C}1\rtimes_{\sigma^\varphi}\mathbb{R}$ be a non-zero projection with $\mathrm{Tr}(q)$ finite. Since $\mathrm{Tr}\!\upharpoonright_Q$ must be semifinite (see e.g.~\cite[Lemma V.7.11]{Takesaki:Book}) and $Q$ diffuse, we may and do assume, by conjugating $Q$ by a unitary, that $q$ falls in $Q$. Remark that $C := Qq$ is also a Cartan subalgebra in $q\widetilde{M}q$ (see e.g.~\cite[Lemma 4.1 (i)]{Voeden:PLMS73} and \cite[Proposition 2.7]{HoudayerRicard:AdvMath11}). Here, we have known by \cite[Theorem 4.1]{Ueda:AdvMath11} that $M = M_c$ is a non-hyperfinite factor of type III, and hence it is standard, see e.g.~\cite[Proposition 2.8]{BoutonnetHoudayerRaum:Preprint12}, that $q\widetilde{M}q$ has no hyperfinite direct summand. Applying the above (I) or (II) to $p:=q$ and $P:=C$ we have $C \preceq_{\widetilde{M}} \mathbb{C}1\rtimes_{\sigma^\varphi}\mathbb{R}$. Then a contradiction will occur once we prove the next general lemma, which holds true for arbitrary von Neumann algebras $M$ and enables us to avoid the `case-by-case' proof of \cite[Theorem D (1)]{HoudayerRicard:AdvMath11} (see the proof of Proposition \ref{P4}).    

\begin{lemma}\label{L3} Let $p \in \widetilde{M}$ be a non-zero projection with $\mathrm{Tr}(p)$ finite, and let $A$ be a MASA in $p\widetilde{M}p$. If either the centralizer $M_\varphi$ is diffuse or $\mathcal{N}_{p\widetilde{M}p}(A)''$ has no type I direct summand, then $A \not\preceq_{\widetilde{M}} \mathbb{C}1\rtimes_{\sigma^\varphi}\mathbb{R}$.  
\end{lemma} 
\begin{proof}
Choose a MASA $D$ in $M_\varphi$. It is plain to see that if $M_\varphi$ is diffuse, then so is $D$.  By \cite[Proposition 2.4]{HoudayerRicard:AdvMath11} (or Proposition \ref{P2}) $B := D\rtimes_{\sigma^\varphi}\mathbb{R} = D\,\bar{\otimes}\,L(\mathbb{R})$ becomes a MASA in $\widetilde{M}$. Suppose, on the contrary, that $A \preceq_{\widetilde{M}} \mathbb{C}1\rtimes_{\sigma^\varphi}\mathbb{R}$ ($\subseteq B$). Since $A$ and $B$ are MASAs in $p\widetilde{M}p$ and $\widetilde{M}$, respectively, \cite[Proposition 2.3]{HoudayerRicard:AdvMath11} with its proof ensures that there exists a non-zero partial isometry $v \in \widetilde{M}$ such that $vv^* \in A$, $v^* v \in B$, $v^* A v = Bv^* v$ and $A vv^* \preceq_{\widetilde{M}} \mathbb{C}1\rtimes_{\sigma^\varphi}\mathbb{R}$. Choose a maximal, orthogonal family of minimal projections $e_1,e_2,\dots$ of $D$ ({\it n.b.}, this family is empty when $M_\varphi$ is diffuse). Set $e_0 := 1 - \sum_{k\geq1} e_k$, and then $De_0$ must be diffuse or $e_0 = 0$. Then two possibilities occur; namely $v(e_k\otimes1)\neq0$ for some $k \geq 1$ or not. We first prove that the former case is impossible thanks to Proposition \ref{P2}, while the latter case can easily be handled thanks to \cite[Proposition 5.3]{HoudayerRicard:AdvMath11}.  

\medskip
Assume that $w:=v(e_k\otimes1)\neq0$ for some $k \geq 1$. In this case, $\mathcal{N}_{p\widetilde{M}p}(A)''$ has no type I direct summand. Remark that $e_k\otimes1 \in D\,\bar{\otimes}\,L(\mathbb{R}) = B$. Thus $w^* Aw = Bw^* w \subseteq w^* w(M\rtimes_{\sigma^\varphi}\mathbb{R})w^* w = w^* w(e_k M e_k \rtimes_{\sigma^{\varphi_{e_k}}}\mathbb{R})w^* w$, where we define $\varphi_{e_k} := \varphi\!\upharpoonright_{e_k M e_k}$ so that $\sigma_t^{\varphi_{e_k}} = \sigma_t^\varphi\!\upharpoonright_{e_k M e_k}$ for every $t \in \mathbb{R}$ because $e_k \in D \subseteq M_\varphi$. Note that $Bw^* w = (De_k\rtimes_{\sigma^{\varphi_{e_k}}}\mathbb{R})w^* w = (\mathbb{C}e_k\,\bar{\otimes}\,L(\mathbb{R}))w^* w = (\mathbb{C}e_k\rtimes_{\sigma^{\varphi_{e_k}}}\mathbb{R})w^* w$. By the hypothesis here $\mathcal{N}_{ww^*\widetilde{M}ww^*}(Aww^*)'' = ww^*(\mathcal{N}_{p\widetilde{M}p}(A)'')ww^*$ (see e.g.~\cite[Proposition 2.7]{HoudayerRicard:AdvMath11}) is non-commutative, and hence the normalizer of $(\mathbb{C}e_k\rtimes_{\sigma^{\varphi_{e_k}}}\mathbb{R})w^* w$ in $w^* w(e_k M e_k \rtimes_{\sigma^{\varphi_{e_k}}}\mathbb{R})w^* w$ does not sit inside $(\mathbb{C}e_k\rtimes_{\sigma^{\varphi_{e_k}}}\mathbb{R})w^* w$ itself. By e.g.~\cite[Lemma 4.1 (i)]{Voeden:PLMS73} one observes that $(e_k M e_k)_{\varphi_{e_k}} = (\mathbb{C}e_k)' \cap (e_k M e_k)_{\varphi_{e_k}} = (De_k)' \cap (e_k M e_k)_{\varphi_{e_k}} = (De_k)' \cap e_k M_\varphi e_k = (D' \cap M_\varphi)e_k = De_k = \mathbb{C}e_k$. Therefore, by Proposition \ref{P2}, $\mathbb{C}e_k\rtimes_{\sigma^{\varphi_{e_k}}}\mathbb{R}$ is a singular MASA in $e_k M e_k \rtimes_{\sigma^{\varphi_{e_k}}}\mathbb{R}$, and thus by e.g.~\cite[Proposition 2.7]{HoudayerRicard:AdvMath11} again so is $(\mathbb{C}e_k\rtimes_{\sigma^{\varphi_{e_k}}}\mathbb{R})w^* w$ in $w^* w(e_k M e_k \rtimes_{\sigma^{\varphi_{e_k}}}\mathbb{R})w^* w$, a contradiction.    

\medskip
We then treat with the remaining case; namely $v(e_k\otimes1)=0$ for all $k \geq 1$, that is, $v^* v \leq e_0 \otimes1$. This case was essentially treated in the proof of \cite[Theorem D (1)]{HoudayerRicard:AdvMath11}. One has $(De_0\,\bar{\otimes}\,\mathbb{C}1_{L(\mathbb{R})})v^* v \subseteq B v^* v = v^* A v \cong_{\mathrm{Ad}v} A vv^* \preceq_{\widetilde{M}} \mathbb{C}1\rtimes_{\sigma^\varphi}\mathbb{R}$ so that $(De_0\,\bar{\otimes}\,\mathbb{C}1_{L(\mathbb{R})})v^* v \preceq_{\widetilde{M}} \mathbb{C}1\rtimes_{\sigma^\varphi}\mathbb{R}$ holds. Set $N:=De_0 \oplus e_0^\perp M e_0^\perp$ of $M$, a diffuse von Neumann subalgebra of $M$. Since $v^* v \leq e_0\otimes1$ and $v^* v \in B=D\,\bar{\otimes}\,L(\mathbb{R})$, one easily sees that $v^* v$ falls into $N' \cap \widetilde{M}$ via the canonical embedding $N \subseteq M \hookrightarrow  M\rtimes_{\sigma^\varphi}\mathbb{R} = \widetilde{M}$. By \cite[Proposition 5.3]{HoudayerRicard:AdvMath11} we obtain that $(De_0\,\bar{\otimes}\,\mathbb{C}1_{L(\mathbb{R})})v^* v\,(= Nv^* v)\,\npreceq_{\widetilde{M}} \mathbb{C}1 \rtimes_{\sigma^{\varphi}}\mathbb{R}$, a contradiction. Hence we are done. 
\end{proof}

\section{Remarks} 

The next free product counterpart of \cite[Theorem D (1)]{HoudayerRicard:AdvMath11} holds thanks to Lemma \ref{L3}.   

\begin{proposition}\label{P4}
If both $M_1,M_2$ are hyperfinite, $M_c$ is of type III$_1$ and $e \in \widetilde{M_c}$ a non-zero finite projection, then the normalizer of any MASA in $e\widetilde{M_c}e$ generates a hyperfinite von Neumann subalgebra.
\end{proposition} 
\begin{proof} We work inside $\widetilde{M}$, and note that $\widetilde{M}_c$ is a direct summand of $\widetilde{M}$ with $1_{\widetilde{M_c}} = 1_{M_c}\otimes 1_{L(\mathbb{R})} \in \mathcal{Z}(\widetilde{M})$. Suppose, on the contrary, that there exists a MASA $C$ in $e\widetilde{M_c}e$ such that $\mathcal{N}_{e\widetilde{M_c}e}(C)''$ is not hyperfinite. The center of $\mathcal{N}_{e\widetilde{M_c}e}(C)''$ sits inside $C$, and thus replacing $e$ with a smaller non-zero projection in $C$ if necessary we may and do also assume that $\mathcal{N}_{e\widetilde{M_c}e}(C)''$ has no hyperfinite direct summand. Since $\widetilde{M_c}$ is a factor, $\mathbb{C}1_{M_c}\rtimes_{\sigma^{\varphi_c}}\mathbb{R}$ diffuse and $\mathrm{Tr}\!\upharpoonright_{\mathbb{C}1_{M_c}\rtimes_{\sigma^{\varphi_c}}\mathbb{R}}$ semifinite, the projection $e$ is equivalent to a non-zero projection $1_{M_c}\otimes e_0 \in \mathbb{C}1_{M_c}\rtimes_{\sigma^\varphi}\mathbb{R}$; hence we may and do further assume that there exists a projection $f = 1\otimes e_0 \in \mathbb{C}1\rtimes_{\sigma^\varphi}\mathbb{R}$ such that $\mathrm{Tr}(f)$ is finite and $e = 1_{M_c}\otimes e_0 = 1_{\widetilde{M_c}}f$. The known fact summarized as (I) in the proof of Theorem \ref{T1} shows $C \preceq_{\widetilde{M}} \mathbb{C}1\rtimes_{\sigma^\varphi}\mathbb{R}$, a contradiction due to Lemma \ref{L3}. 
\end{proof} 
   
Proposition \ref{P2} and \cite[Remark 5.4]{HoudayerRicard:AdvMath11} precisely show that the von Neumann subalgebras generated by the quasi-normalizer and the (groupoid) normalizer of $N\rtimes_{\sigma^\psi}\mathbb{R} \supset \mathbb{C}1\rtimes_{\sigma^\psi}\mathbb{R}$ are different in general. Thus one may expect that the continuous core of any `type III$_1$ free product factor' has no regular diffuse hyperfinite von Neumann subalgebra. This is unclear at the moment of this writing, but the next weaker assertion follows directly from \cite{HoudayerRicard:AdvMath11},\cite{BoutonnetHoudayerRaum:Preprint12}. 

\begin{proposition}\label{P5} 
If $M_c$ is of type III$_1$ and $e \in \widetilde{M_c}$ a non-zero finite projection, then $e\widetilde{M_c}e$ has no regular, hyperfinite type II$_1$ von Neumann algebra. 
\end{proposition}
\begin{proof} We may and assume that $M = M_c$, $e \in \mathbb{C}1\rtimes_{\sigma^{\varphi}}\mathbb{R}$, and that $M_i$ has no hyperfinite direct summand as long as it is not hyperfinite. If $P$ is a regular hyperfinite type II$_1$ von Neumann algebra of $e\widetilde{M}e$, then as in Theorem \ref{T1} one gets $P \preceq_{\widetilde{M}} \mathbb{C}1\rtimes_{\sigma^\varphi}\mathbb{R}$, which is impossible.   
\end{proof}

A free product counterpart of \cite[Theorem 1.2]{Houdayer:JIMJ10} was also implicitly shown in \cite{ChifanHoudayer:DukeMathJ10} as follows. 

\begin{proposition}\label{P6} If both $M_1,M_2$ are hyperfinite, $M_c$ is of type III$_1$ and $e \in \widetilde{M_c}$ a non-zero finite projection, then the relative commutant of any type II$_1$ von Neumann subalgebra in $e\widetilde{M}_ce$ is hyperfinite.
\end{proposition} 
\begin{proof} Suppose, on the contrary, that there exists a type II$_1$ von Neumann subalgebra $N$ of $e\widetilde{M_c}e = e\widetilde{M}e$ (since $\widetilde{M_c}$ is a direct summand of $\widetilde{M}$) such that $N'\cap e\widetilde{M}e$ is not hyperfinite. One can choose a non-zero projection $z \in \mathcal{Z}(N'\cap e\widetilde{M}e)$ so that $P := (Nz)'\cap z\widetilde{M}z = (N'\cap e\widetilde{M}e)z$ (see e.g.~\cite[Lemma 4.1 (ii)]{Voeden:PLMS73}) has no hyperfinite direct summand.  By \cite[Theorem 4.2]{ChifanHoudayer:DukeMathJ10} $P'\cap z\widetilde{M}z \preceq_{\widetilde{M}} \widetilde{M_{i_0}}$ for some $i_0$, but $Nz \subseteq P'\cap z\widetilde{M}z \preceq_{\widetilde{M}} \mathbb{C}1\rtimes_{\sigma^\varphi}\mathbb{R}$ is impossible; hence by \cite[Proposition 2.7]{BoutonnetHoudayerRaum:Preprint12} $P \subseteq \mathcal{N}_{z\widetilde{M}z}(P'\cap z\widetilde{M}z)'' \preceq_{\widetilde{M}} \widetilde{M_{i_0}}$, a contradiction. 
\end{proof}

\section*{Acknowledgement} The most crucial part of this work was carried out during the author's stay at Hokkaido University in Dec.~4--7, 2013. We thank Reiji Tomatsu for his hospitality.

\end{document}